\documentclass[12pt,a4paper]{amsart}
\usepackage{amsmath,amstext,amssymb,amsthm,amsfonts}
\usepackage{xcolor}

\newcommand{\nc}{\newcommand}

\nc{\one}{\mbox{\bf 1}}
\nc{\invtensor}{\underset{\leftarrow}{\otimes}}
\nc{\const}{\operatorname{const}}

\nc{\Ad}{\operatorname{Ad}}
\nc{\ad}{\operatorname{ad}}
\nc{\ev}{\operatorname{ev}}
\nc{\tr}{\operatorname{tr}}

\nc{\Gr}{\mathscr{K}}
\nc{\rGr}{\operatorname{rGr}}
\nc{\atyp}{\operatorname{atyp}}
\nc{\tw}{\operatorname{tw}}
\nc{\rank}{\operatorname{rank}}
\nc{\corank}{\operatorname{corank}}
\nc{\codim}{\operatorname{codim}}
\nc{\sdim}{\operatorname{sdim}}
\nc{\mult}{\operatorname{mult}}
\nc{\ds}{\operatorname{ds}}
\nc{\tail}{\operatorname{tail}}
\nc{\howl}{\operatorname{howl}}
\nc{\triv}{\operatorname{triv}}
\nc{\spn}{\operatorname{span}}
\nc{\Sym}{\operatorname{Sym}}
\nc{\Core}{\operatorname{Core}}
\nc{\id}{\operatorname{id}}
\nc{\Id}{\operatorname{Id}}
\nc{\Ree}{\operatorname{Re}}
\nc{\hi}{\operatorname{hi}}
\nc{\htt}{\operatorname{ht}}
\nc{\at}{\operatorname{at}}
\nc{\str}{\operatorname{str}}
\nc{\Iso}{\operatorname{Iso}}
\nc{\Ker}{\operatorname{Ker}}
\nc{\rker}{\operatorname{rKer}}
\nc{\im}{\operatorname{Im}}
\nc{\osp}{\mathfrak{osp}}
\nc{\sgn}{\operatorname{sgn}}
\nc{\F}{\operatorname{F}}
\nc{\Mod}{\operatorname{Mod}}
\nc{\DS}{\operatorname{DS}}
\nc{\Soc}{\operatorname{Soc}}
\nc{\Inj}{\operatorname{Inj}}
\nc{\Hom}{\operatorname{Hom}}
\nc{\End}{\operatorname{End}}
\nc{\supp}{\operatorname{supp}}
\nc{\smult}{\operatorname{smult}}
\nc{\Dyn}{\operatorname{Dyn}}
\nc{\Card}{\operatorname{Card}}
\nc{\Ann}{\operatorname{Ann}}
\nc{\Arc}{\operatorname{Arc}}
\nc{\arc}{\operatorname{arc}}
\nc{\Ind}{\operatorname{Ind}}
\nc{\Coind}{\operatorname{Coind}}
\nc{\wt}{\operatorname{wt}}
\nc{\hwt}{\operatorname{hwt}}
\nc{\rk}{\operatorname{rank}}
\nc{\ch}{\operatorname{ch}}
\nc{\sch}{\operatorname{sch}}
\nc{\mdim}{\operatorname{mdim}}
\nc{\Stab}{\operatorname{Stab}}
\nc{\defect}{\operatorname{defect}}

\nc{\Irr}{\operatorname{Irr}}
\nc{\Spec}{\operatorname{Spec}}
\nc{\Res}{\operatorname{Res}}
\nc{\res}{\operatorname{res}}
\nc{\Aut}{\operatorname{Aut}}
\nc{\Ext}{\operatorname{Ext}}
\nc{\Prec}{\operatorname{Prec}}
\nc{\Fract}{\operatorname{Fract}}
\nc{\gr}{\operatorname{gr}}
\nc{\vol}{\operatorname{vol}}
\nc{\diag}{\operatorname{diag}}
\nc{\deff}{\operatorname{def}}
\nc{\core}{\operatorname{core}}
\nc{\HC}{\operatorname{HC}}
\nc{\Ch}{\operatorname{Ch}}
\nc{\Sch}{\operatorname{Sch}}
\nc{\dpth}{\operatorname{dpth}}
\nc{\Pol}{\operatorname{Pol}}

\nc{\pari}{\operatorname{par}}
\nc{\pos}{\operatorname{pos}}

\nc{\Cl}{\mathcal{C}\ell}

\nc{\wdchi}{\widetilde{\chi}}
\nc{\wdH}{\widetilde{H}}
\nc{\wdN}{\widetilde{N}}
\nc{\wdM}{\widetilde{M}}
\nc{\wdO}{\widetilde{O}}
\nc{\wdR}{\widetilde{R}}

\nc{\wdV}{\widetilde{V}}

\nc{\wdC}{\widetilde{C}}

\nc{\zero}{\operatorname{zero}}
\nc{\nonzero}{\operatorname{nonzero}}

\nc{\Nonzero}{\operatorname{Nonzero}}

\nc{\diam}{\operatorname{diam}}

\nc{\Obj}{\operatorname{Obj}}
\nc{\Dglie}{\operatorname{{\mathcal D}glie}}
\nc{\Fin}{\operatorname{{\mathcal F}in}}
\nc{\pr}{\operatorname{pr}}
\nc{\Adm}{\operatorname{\mathcal{A}dm}}

\nc{\fg}{\mathfrak g}

\nc{\Sg}{{\cS(\fg)}}
\nc{\Shg}{{\cS(\fhg)}}
\nc{\Ug}{{\cU(\fg)}}
\nc{\Uhg}{{\cU(\fhg)}}
\nc{\Sh}{{\cS(\fh)}}
\nc{\Uh}{{\cU(\fh)}}
\nc{\Uhh}{{\cU(\fhh)}}
\nc{\Zg}{{{\mathcal{Z}}(\fg)}}

\nc{\Vir}{{\mathcal{V}ir}}
\nc{\NS}{{\mathcal{N}S}}

\newcommand{\Sk}{Sk}
\newcommand{\Sp}{Sp}

\newcommand{\ree}{re}
\newcommand{\ima}{im}
\newcommand{\is}{iso}

\nc{\Up}{{\mathcal U}({\mathfrak p})}
\nc{\Ah}{{\mathcal A}({\mathfrak h})}
\nc{\Ag}{{\mathcal A}({\mathfrak g})}
\nc{\Ap}{{\mathcal A}({\mathfrak p})}
\nc{\Zp}{{\mathcal Z}({\mathfrak p})}
\nc{\cR}{\mathcal R}
\nc{\cS}{\mathcal S}
\nc{\cP}{\mathcal P}
\nc{\cT}{\mathcal{T}}
\nc{\CC}{\mathcal C}
\nc{\cA}{\mathcal A}
\nc{\cV}{\mathcal V}
\nc{\cU}{\mathcal U}
\nc{\cZ}{\mathcal Z}
\nc{\cM}{\mathcal M}
\nc{\cL}{\mathcal L}
\nc{\cF}{\mathcal F}

\nc{\cB}{\mathcal{B}}

\nc{\fo}{\mathfrak o}
\nc{\fspo}{\mathfrak{spo}}
\nc{\fa}{\mathfrak a}

\nc{\CO}{\mathcal O}
\nc{\CR}{\mathcal R}

\nc{\cK}{\mathcal{K}}
\nc{\cW}{\mathcal{W}}
\nc{\bM}{\mathbf{M}}
\nc{\bL}{\mathbf{L}}
\nc{\bN}{\mathbf{N}}

\nc{\zq}{\mathpzc q}

\nc{\fl}{\mathfrak l}
\nc{\fn}{\mathfrak n}
\nc{\fm}{\mathfrak m}
\nc{\fp}{\mathfrak p}
\nc{\fh}{\mathfrak h}
\nc{\ft}{\mathfrak t}
\nc{\fk}{\mathfrak k}
\nc{\fb}{\mathfrak b}
\nc{\fs}{\mathfrak s}
\nc{\fB}{\mathfrak B}

\def\F{{\mathbb F}}

\newcommand{\bt}{\begin{theorem}}
\newcommand{\et}{\end{theorem}}

\newcommand{\be}{\begin{equation}}
\newcommand{\ee}{\end{equation}}

\newcommand{\bpf}{\begin{proof}}
\newcommand{\epf}{\end{proof}}

\newcommand{\bi}{\begin{enumerate}}
\newcommand{\ei}{\end{enumerate}}
\newcommand{\bl}{\begin{lem}}
\newcommand{\el}{\end{lem}}

\nc{\vareps}{\varepsilon}
\nc{\varesp}{\varepsilon}
\nc{\veps}{\varepsilon}

\nc{\fsl}{\mathfrak{sl}}
\nc{\fgl}{\mathfrak{gl}}
\nc{\fso}{\mathfrak{so}}
\nc{\fosp}{\mathfrak{osp}}
\nc{\fsp}{\mathfrak{sp}}
\nc{\fq}{\mathfrak q}
\nc{\fsq}{\mathfrak{sq}}
\nc{\fpsq}{\mathfrak{psq}}
\nc{\fpq}{\mathfrak{pq}}

\nc{\fhg}{\hat{\fg}}
\nc{\fhn}{\hat{\fn}}
\nc{\fhh}{\hat{\fh}}
\nc{\fhb}{\hat{\fb}}
\nc{\hrho}{\hat{\rho}}

\nc{\hsl}{\hat{\fsl}}
\nc{\fpo}{\mathfrak{po}}
\nc{\dirlim}{\underset{\rightarrow}{\lim}\,}
\nc{\nen}{\newenvironment}
\nc{\ol}{\overline}
\nc{\ul}{\underline}
\nc{\ra}{\rightarrow}
\nc{\lra}{\longrightarrow}
\nc{\Lra}{\Longrightarrow}
\nc{\bo}{\bar{1}}
\nc{\Lla}{\Longleftarrow}

\nc{\Llra}{\Longleftrightarrow}

\nc{\thla}{\twoheadleftarrow}

\nc{\lang}{(}
\nc{\rang}{)}

\nc{\hra}{\hookrightarrow}
\nc{\simO}{\overset{\CO}{\sim}}
\nc{\iso}{\overset{\sim}{\lra}}

\nc{\ssubset}{\underset{\not=}{\subset}}

\nc{\vac}{|0\rangle}

\nc{\simka}{{\ \scriptscriptstyle _{{\sim}}^\text{\tiny{k}}\ }}

\nc{\Thm}[1]{Theorem~\ref{#1}}
\nc{\Prop}[1]{Proposition~\ref{#1}}
\nc{\Lem}[1]{Lemma~\ref{#1}}
\nc{\Cor}[1]{Corollary~\ref{#1}}
\nc{\Conj}[1]{Conjecture~\ref{#1}}
\nc{\Claim}[1]{Claim~\ref{#1}}
\nc{\Defn}[1]{Definition~\ref{#1}}
\nc{\Exa}[1]{Example~\ref{#1}}
\nc{\Rem}[1]{Remark~\ref{#1}}
\nc{\Note}[1]{Note~\ref{#1}}
\nc{\Quest}[1]{Question~\ref{#1}}
\nc{\Hyp}[1]{Hypoth\`ese~\ref{#1}}
\nen{thm}[1]{\label{#1}{\bf Theorem.\ } \em}{}
\nen{prop}[1]{\label{#1}{\bf Proposition.\ } \em}{}
\nen{lem}[1]{\label{#1}{\bf Lemma.\ } \em}{}
\nen{cor}[1]{\label{#1}{\bf Corollary.\ } \em}{}
\nen{conj}[1]{\label{#1}{\bf Conjecture.\ } \em}{}

\nen{claim}[1]{\label{#1}{\bf Claim.\ } \em}{}

\nen{defn}[1]{\label{#1}{\bf Definition.\ } }{}
\nen{exa}[1]{\label{#1}{\bf Example.\ } }{}
\nen{rem}[1]{\label{#1}{\em Remark.\ } }{}
\nen{note}[1]{\label{#1}{\em Note.\ } }{}
\nen{exer}[1]{\label{#1}{\em Exercise.\ } }{}
\nen{sket}[1]{\label{#1}{\em Sketch of proof.\ } }{}
\nen{quest}[1]{\label{#1}{\bf Question.\ } \em}{}

\nen{hyp}[1]{\label{#1}{\bf Hypoth\`ese.\ } \em}{}
\begin{document}
\setcounter{section}{-1}
\setcounter{tocdepth}{1}

%
\begin{center}
{\large
Irreducibility of Verma modules over symmetrizable
Kac--Moody superalgebras
\par}

\vspace{1.5em}

{\large Maria Gorelik$^{1}$\par}

\vspace{1em}

{\small
$^{1}$Department of Mathematics and Computer Science,
Weizmann Institute of Science, Rehovot 7610001, Israel.
\par}

\vspace{1em}

{\small
\textbf{Corresponding author:} Maria Gorelik\\
E-mail: maria.gorelik@gmail.com
\par}
\end{center}

\bigskip

\vspace{1em}

{\small\textbf{Abstract:}
We establish a criterion for the irreducibility of Verma modules over
symmetrizable Kac--Moody superalgebras. Our main tool is the Weyl-Kac
denominator identity.\par}

%
%
%
%
%
%
%


%






\section{Introduction}

The present paper contributes to the representation theory of Kac--Moody superalgebras $\fg:=\fg(A,\tau)$, which are constructed from an $\ell\times\ell$  Cartan matrix $A$ and a subset $\tau\subset\{1,\ldots,\ell\}$ by a procedure analogous to that described in~\cite{Kbook}, Chapter I. 
These Lie superalgebras were introduced in~\cite{Ksuper}.
The subset $\tau$ determines the odd generators. Ordinary Kac--Moody algebras are recovered as the special case $\tau=\emptyset$. 

The precise definition of Kac--Moody superalgebras (see Section~\ref{sect:KM}) is more involved than that of Kac--Moody algebras. The main result of the paper is an irreducibility criterion for a Verma module over symmetrizable Kac-Moody superalgebra.

By  Hoyt-Serganova classification~\cite{Hoyt},\cite{Sint},
 indecomposable symmetrizable Kac--Moody superalgebras fall into three classes:
\begin{itemize}
\item[I.] finite-dimensional simple Lie algebras and their super analogues;

\item[II.] affine Lie superalgebras, which are subquotients of central extensions of twisted loop algebras associated with Lie superalgebras from Class I;

\item[III.] anisotropic symmetrizable Kac--Moody superalgebras $\fg(A,\tau)$ (studied in~\cite{K78}), for which 
$a_{ii}\neq 0$, 
$\frac{2a_{ij}}{a_{ii}}\in\mathbb Z_{\leq 0}$ if $i\neq j$,
and $\frac{a_{ij}}{a_{ii}}\in\mathbb Z$ whenever $i\in\tau$. 
\end{itemize}

A Verma module $M(\lambda)$  is irreducible if and only if all Shapovalov determinants $\det {S}_{\nu}$ are nonzero at $\lambda$. The Shapovalov determinants, introduced in~\cite{Shap}, are polynomials that admit linear factorization  (under our assumption that $\fg(A,\tau)$ is symmetrizable):
$$\det S_{\nu}=\prod_{m=1}^{\infty}\prod_{\alpha\in\Delta^+}
H_{m\alpha}^{d_{m\alpha}(\nu)},
$$
where $\Delta^+$ is the set of positive roots  and
$$\begin{array}{l}
H_{m\alpha}(\lambda):=(\lambda+\rho,m\alpha)-\frac{(m\alpha,m\alpha)}{2},\\
d_{m\alpha}(\nu):=(-1)^{(m+1)p(\alpha)}\dim\fg_\alpha K(\nu-m\alpha),
\end{array}$$
where $K:\mathbb{Z}\Delta\to\mathbb{Z}_{\geq 0}$ stands for  the Kostant partition function and $p(\alpha)$ stands for the parity of root.

One has  $(-1)^{(m+1)p(\alpha)}d_{m\alpha}(\nu)\geq 0$  
and this number is positive for sufficiently large $\nu$.

If $\fg$ is a Kac--Moody algebra, $p(\alpha)=0$.
This yields the well-known irreducibility criterion: the Verma module $M(\lambda)$ is irreducible if and only if
$$
H_{m\alpha}(\lambda)\not= 0
$$
for all pairs $(m,\alpha)$, see~\cite{Shap},\cite{KK}.

For Kac--Moody superalgebras, cancellations between factors may occur: if $\alpha$ is an odd root and $2\alpha$ is a root, then the same factor $H_{2m\alpha}$ arises from the pairs $(2m,\alpha)$ and $(m,2\alpha)$, with contributions of opposite signs. This leads naturally to the problem of determining the sign of
$$
\dim\fg_{2\alpha}-\dim\fg_\alpha
$$
for odd roots $\alpha$.

For a real root $\alpha$ one has 
$\dim\fg_\alpha=1$, with
$\dim\fg_{2\alpha}=1$ if $\alpha$ is  odd and non-isotropic, and $\dim\fg_{2\alpha}=0$ 
otherwise.
Thus, it remains to consider odd imaginary roots.
There are no imaginary roots in Class I, while in Class II all imaginary roots are even, with the exception of one series. For the odd imaginary roots in this exceptional series, we have
$1=\dim\fg_\alpha<\dim\fg_{2\alpha}$.

It remains, therefore, to treat Class III. In this case the root multiplicities are given only implicitly by the Weyl-Kac denominator identity. We use this  identity  as a tool for comparing root multiplicities: by comparing the identities for the ``sibling'' Kac--Moody superalgebras $\fg(A,\tau)$ and $\fg(A,\emptyset)$, we obtain
\begin{equation}\label{eq:1>2}
\dim\fg_{2\gamma}>\dim\fg_\gamma\ \text{for every odd imaginary root $\gamma$.}\end{equation}

Note that, by above, (\ref{eq:1>2}) holds for every symmetrizable Kac--Moody superalgebra. The symmetrizability assumption is essential here: the  inequality fails for the non-symmetrizable affine Kac--Moody superalgebra $\fq_n^{(2)}$, where  the minimal imaginary root $\delta$ is odd and
$\dim\fg_{2\delta}=\dim\fg_\delta=n-1$.

The  inequality (\ref{eq:1>2}) allows us to control the cancellations among the linear factors in the Shapovalov determinants and yields the following irreducibility criterion.

\subsection{}
\begin{cor}{cor:1}
Let $\fg$ be a symmetrizable Kac-Moody superalgebra.
A Verma module $M(\lambda)$ is irreducible if and only if
$$
2(\lambda+\rho,\alpha)\notin
\mathbb Z_{>0}(\alpha,\alpha)
$$
for every even positive real root $\alpha$\footnote{In our convention
for a real root $\alpha$, $\alpha/2$ is not a root}, 
$$
2(\lambda+\rho,\alpha)\notin
(2\mathbb Z_{\geq0}+1)(\alpha,\alpha)
$$
for every odd positive real root $\alpha$, and
$$2(\lambda+\rho,\alpha)\not=(\alpha,\alpha)
$$
 for every positive imaginary root $\alpha$.
\end{cor}

 For Classes I and II this criterion specializes to a well-known irreducibility criterion.

 As a further application, in~\ref{blocks} we describe the irreducible modules in the blocks of category $\CO$ via the corrected criterion, which slightly corrects~\cite{GSsnow}, 1.13, where an incorrect irreducibility criterion for Verma modules was used.

\subsection{}
The paper is organized as follows. In Section~\ref{sect:KM}, we recall the definition and basic properties of Kac-Moody superalgebras, including reflexions, the skeleton, and real and imaginary roots, and briefly review their classification. In the next section, we introduce sibling Kac--Moody superalgebras and, by comparing their Weyl denominator identities, establish the root-multiplicity inequality needed for the irreducibility criterion. We then apply this inequality to the Shapovalov determinant for symmetrizable Kac--Moody superalgebras and derive the irreducibility criterion for Verma modules. Then we describe the irreducible modules in the blocks of category $\CO$.
In the final section, we discuss further properties of Shapovalov determinants, including their leading terms and real-root factors, and briefly review what is known in the non-symmetrizable cases.

\section*{Acknowledgments}
I am grateful to Vladimir Hinich, 
Victor Kac and Vera Serganova for helpful discussions and comments. I would also like to thank the members of the STARS seminar for their interest and useful feedback during the presentation of these results.

I used ChatGPT (OpenAI) to assist with polishing the English and improving the exposition of the manuscript. I thank Claude (Anthropic) for a careful reading of an earlier version of the manuscript. It pointed out several inaccuracies, including the omission of imaginary roots in Corollary 4.3, the exponent of the Kostant-type function in 4.2.1, and the treatment of isotropic imaginary roots in 4.4.2. All mathematical statements and proofs were checked by the author, who is responsible for any remaining errors.

This research was 
supported by the Minerva Foundation with funding from the Federal German Ministry for Education and Research.

\section{Kac-Moody superalgebras $\fg(A,\tau)$}\label{sect:KM}

Our base field is $\mathbb{C}$. In this section, we introduce the main definitions, including that of a symmetrizable Kac--Moody superalgebra. Our definition is a slightly simplified version of the one given in~\cite{GHS}, except that we allow simple isotropic roots; in particular, it includes the semidirect product of a Heisenberg Lie algebra and its derivation.

\subsection{Contragredient superalgebras $\fg(A,\tau)$}

The construction below is similar to that in~\cite{Kbook}, Chapter I.

To define a contragredient superalgebra $\fg(A,\tau)$ of rank $\ell$, we start with a complex $\ell\times\ell$ Cartan matrix $A$, a subset $\tau\subset\{1,\ldots,\ell\}$, a finite-dimensional vector space $\fh$, and two families of linearly independent vectors
$\{\alpha_i^\vee\}_{i=1}^{\ell}\subset\fh$ and
$\Sigma=\{\alpha_i\}_{i=1}^{\ell}\subset\fh^*$
satisfying
$\langle\alpha_i^\vee,\alpha_j\rangle=a_{ij}$.
We call $(A,\tau)$ a {\em Cartan datum}.

Following~\cite{GHS}, we consider the ``half-baked'' Lie superalgebra $\tilde{\fg}$ generated by $\fh$ and the Chevalley generators $e_i,f_i$, $i=1,\ldots,\ell$, which are odd for $i\in\tau$ and even otherwise, subject to the relations
$$
\begin{array}{l}
[h,e_i]=\alpha_i(h)e_i,\qquad
[h,f_i]=-\alpha_i(h)f_i,\qquad
[\fh,\fh]=0,\qquad 
[e_i,f_j]=\delta_{ij}\alpha_i^\vee
\end{array}
$$
for $i,j=1,\ldots,\ell$.

The Lie superalgebra $\tilde{\fg}$ has a unique maximal ideal with zero intersection with $\fh$, and the contragredient Lie superalgebra $\fg:=\fg(A,\tau)$ is the quotient of $\tilde{\fg}$ by this ideal.

\subsubsection{}
Denote by $\tilde{\fn}^-$ (resp., $\tilde{\fn}$) the Lie subalgebra of $\tilde{\fg}$ generated by the $f_i$'s (resp., the $e_i$'s), $i=1,\ldots,\ell$. Arguing as in~\cite{Kbook}, Chapter I, one shows that $\tilde{\fn}^-$ and $\tilde{\fn}$ are free Lie superalgebras on these sets of generators and that
$\tilde{\fg}=\tilde{\fn}^-+\fh+\tilde{\fn}$.
This induces the triangular decomposition
$\fg=\fn^-+\fh+\fn$.
We set
$$
\begin{array}{l}
\Delta:=\{\alpha\in\fh^*|\ \alpha\neq0,\ \fg_\alpha\neq0\},\ \ \ 
\Delta^+:=\{\alpha\in\fh^*|\ \fn_\alpha\neq0\}.
\end{array}
$$

Then
$$
\Delta=-\Delta^+\coprod\Delta^+.
$$

\subsubsection{Parity}\label{parity}
For $\alpha_i\in\Sigma $ we set 
$p(\alpha_i)=\ol{0}$ if $i\not\in\tau$ and  $p(\alpha_i)=\ol{1}$
otherwise. Extend this to the group homomorphism
$p: \mathbb{Z}\Sigma\to \mathbb{Z}_2$ we obtain
$\Delta=\Delta_{\ol{0}}\coprod \Delta_{\ol{1}}$, where
$$\Delta_{\ol{i}}=\{\alpha\in\Delta|\ p(\alpha)=\ol{i}\}=
\{\alpha\in\Delta|\ \fg_\alpha\cap \fg_{\ol{i}}\not=0\}.$$

\subsubsection{Weyl vector}
A vector $\rho\in\fh^*$ satisfying 
$2\langle\rho,\alpha^{\vee}\rangle=\langle\alpha,\alpha^{\vee}\rangle$
is called a {\em Weyl vector}.

\subsubsection{Symmetrizability}
We have $\fg(A,\tau)=\fg(DA,\tau)$ for any
 invertible diagonal matrix $D$. 
The Cartan datum $(A,\tau)$ and the contragredient superalgebra $\fg(A,\tau)$
are called {\em symmetrizable} if $DA$ is symmetric for some 
 invertible diagonal matrix $D$. 
In this case, $\fg(A,\tau)$  admits a non-degenerate supersymmetric invariant bilinear form, which 
induces a non-degenerate symmetric bilinear form $(-,-)$ on $\fh^*$.

\subsection{Algebra $\fg(\alpha)$}\label{galpha}
For each simple root $\alpha=\alpha_i\in\Sigma$, the spaces $\fg_{\pm\alpha}$ and $[\fg_\alpha,\fg_{-\alpha}]$ are one-dimensional, and $[\fg_\alpha,\fg_{-\alpha}]$ is spanned by $\alpha^{\vee}:=\alpha^{\vee}_i$. If $\langle\alpha^\vee,\alpha\rangle=0$, we call $\alpha$ {\em isotropic}; otherwise, we call $\alpha$ {\em non-isotropic}.

Let $\fg(\alpha)$ be the subalgebra generated by $\fg_{\pm\alpha}$. One has
\begin{equation}\label{eq:galpha}
\fg(\alpha)\cong
\left\{
\begin{array}{lll}
\fsl_2 & & \text{if $\alpha$ is even and non-isotropic},\\
\mathfrak{osp}(1|2) & & \text{if $\alpha$ is odd and non-isotropic},\\
\text{Heisenberg Lie algebra} & & \text{if $\alpha$ is even and isotropic},\\
\fsl(1|1):=[\fgl(1|1),\fgl(1|1)] & & \text{if $\alpha$ is odd and isotropic}.
\end{array}
\right.
\end{equation}

\subsubsection{}
Assume that
\begin{equation}\label{eq:cond0}
a_{ij}=0\quad\Longrightarrow\quad a_{ji}=0.
\end{equation}
(If $(A,\tau)$ is symmetrizable, then $A$ satisfies this assumption.)

Using~\cite{Kbook}, Proposition 3.9, one readily sees that $\fg(\alpha_i)$ acts locally finitely in the adjoint representation under the following conditions:
\begin{equation}\label{eq:cond1}
\begin{array}{lcl}
a_{ii}\neq 0,\ i\notin\tau
&\Longrightarrow&
\displaystyle\frac{2a_{ij}}{a_{ii}}\in\mathbb Z_{\leq0}\\
a_{ii}\neq0,\ i\in\tau
&\Longrightarrow&
\displaystyle\frac{2a_{ij}}{a_{ii}}\in2\mathbb Z_{\leq0};\\
a_{ii}=0,\ i\notin\tau
&\Longrightarrow&
a_{ij}=0
\end{array}
\end{equation}
for all $j\neq i$.

From now on, we assume that the Cartan datum satisfies~(\ref{eq:cond0}) and~(\ref{eq:cond1}).

\subsubsection{Reflections}\label{reflections}
If $\alpha\in\Sigma$ is non-isotropic,
 we define the reflection $s_{\alpha}\in GL(\fh^*)$  by  the  formula:
 
$$s_{\alpha}\mu=\mu-\frac{2\langle \mu, \alpha^{\vee}\rangle}{\langle \alpha, \alpha^{\vee}\rangle} \text{ for }\mu\in\fh^*.$$

By~\cite{Kbook}, Lemma 3.8, in this case
$\fg(A,\tau)$ admits an automorphism $\phi_i$
such that $\phi_i(\fg_{\gamma})=\fg_{s_{\alpha}\gamma}$
for all $\gamma\in\Delta$ ($\phi_i:=\exp(f'_i)\exp(-e'_i)\exp(f'_i)$, where $e'_i,f'_i$ are Chevalley generators of
$\fsl_2\subset \fg(\alpha)$).

\subsubsection{Reflexions}\label{reflexions}
For each $\beta\in\Sigma$, the contragredient
Lie superalgebra $\fg(A,\tau)$ is isomorphic 
to a contragredient Lie superalgebra $\fg(A',\tau')$
via the isomorphism 
$$\iota:\fg(A',\tau')\iso \fg(A,\tau)$$
such that $\iota(\fh')=\fh$ and 
$\Delta^{'+}\setminus \mathbb{Z}_{<0}\beta=\Delta^+\setminus \mathbb{Z}_{>0}\beta$.
We denote the new set of simple roots by $r_{\beta}\Sigma$
and set
$r_i(A,\tau):=(A',\tau')$.

If $\beta=\alpha_i$ is non-isotropic, then 
 $\iota(\fg(A',\tau'))$ is 
given by  $e'_i:=\phi_i(e_i)$
and $f'_i:=\phi_i(f_i)$. In this case $r_{\beta}\Sigma:=s_{\alpha_i}\Sigma$
and 
$r_i(A',\tau')=(A,\tau)$. 

However, if $\alpha_i$ is  isotropic,
then $r_i(A',\tau')\not=(A,\tau)$; we will not give formulae for this case except for the formula
$$r_{\alpha_i}\Sigma=\{-\alpha_i\}\cup \{\alpha_j|\ a_{ij}=0, j\not=i\}\cup \{\alpha_j+\alpha_i|\ a_{ij}\not=0\}.$$
(If $\alpha_i$ is even isotropic, then the conditions~(\ref{eq:cond1}) give $r_{\alpha_i}\Sigma=
(\Sigma\setminus\{\alpha_i\})\cup
\{-\alpha_i\}$).

If the new Cartan datum  $r_i(A,\tau)$ satisfies 
the conditions~(\ref{eq:cond0}), (\ref{eq:cond1}),
we can define $r_jr_i(A,\tau)$.
It is easy to see that  $-\beta\in r_{\beta}\Sigma$ and
$r_{-\beta}r_{\beta}\Sigma=\Sigma$ with
 $r_ir_i(A,\tau)=(A,\tau)$. 
We call $\fg(A,\tau)$ a {\em Kac-Moody superalgebra}
if all Cartan data $r_{i_1}\ldots r_{i_s}(A,\tau)$
satisfy the conditions~(\ref{eq:cond0}), (\ref{eq:cond1}).

\subsubsection{Remark}
In the classical definition,  the case
of Heisenberg Lie algebras is excluded
by assuming $a_{ii}\not=0$ for $i\not\in\tau$.

\subsection{Skeleton and real roots}
Let  $\fg(A,\tau)$ be a Kac-Moody superalgebra.
Then $r_{i_1}\ldots r_{i_s}\Sigma$ is well-defined
for any sequence of $1\leq i_1,\ldots,i_s\leq \ell$
and we define the skeleton $\Sk$
as a graph with vertices  $r_{i_1}\ldots r_{i_s}\Sigma$ 
and edges $r_{\beta}: \Sigma\to r_{\beta}\Sigma$.
This graph is studied in~\cite{GHS}.

We say that $\alpha$ is a {\em real root} if $\alpha\in\Sigma'$ for some $\Sigma'\in\Sk$; we denote the set of real roots by $\Delta^{\ree}$. For a real root $\alpha$, we define $\fg(\alpha)$ as above; in particular, the classification~(\ref{eq:galpha}) continues to hold. We call a real root $\alpha$ isotropic if $\fg(\alpha)\cong\fsl(1|1)$ or $\fg(\alpha)$ is a Heisenberg Lie algebra, and non-isotropic otherwise. 

If $\alpha$ is non-isotropic, we define the reflection $s_\alpha$ as above.
The subgroup of $GL(\fh^*)$ generated by the elements $s_\alpha$, $\alpha\in\Delta^{\ree}\setminus\Delta^{\is}$, is called the Weyl group and is denoted by $W$. We denote by $\Delta^{\is}$ the set of isotropic real roots.

We call a root $\alpha\in\Delta$ {\em imaginary} if neither $\alpha$ nor $\alpha/2$ is real, and denote the set of imaginary roots by $\Delta^{\ima}$. Thus, with our convention, if $\alpha$ is an odd non-isotropic real root, then $2\alpha$ is neither real nor imaginary. This convention differs slightly from some conventions in the literature,
in which $2\alpha$ is regarded as a real root.

\subsection{Classification}
We call $(A,\tau)$ {\em anisotropic} if $a_{ii}\not=0$
for all $i$. If $(A,\tau)$ is anisotropic, then
$\fg(A,\tau)$ is a Kac-Moody superalgebra if and only if $(A,\tau)$
satisfies~(\ref{eq:cond0}), (\ref{eq:cond1})
(since $r_i(A,\tau)=(A,\tau)$). 

We call $(A,\tau)$ and $\fg(A,\tau)$ {\em decomposable}
if, after
a permutation of indices, $A=\begin{pmatrix} A_1 &0 \\0 &A_2\end{pmatrix}$  and {\em indecomposable}
otherwise. Every Kac-Moody superalgebra is a direct product
of indecomposable ones. The indecomposable  Kac-Moody superalgebras
are naturally divided into the following classes:
\begin{itemize}
\item{ (FIN)} finite-dimensional;
\item{(AFF)} infinite-dimensional of finite Gelfand-Kirillov
dimension;
\item{(IND)} of infinite Gelfand-Kirillov dimension.
\end{itemize}

In terms of imaginary roots, 
$\Delta^{\ima}$ is empty for (Fin), all imaginary roots are proportional in case (Aff),
and $\Delta^{\ima}$ contains non-proportional imaginary roots
in (Ind), see~\cite{GHS}, Section 7. In all cases, 
if $\delta\in\Delta^{\ima}$, then $m\delta\in\Delta^{\ima}$ for 
all non-zero integral $m$, 
see~\cite{Kbook}, Section 5 (and~\cite{Shay} for $\tau\not=\emptyset$).

In case (FIN), every Kac-Moody superalgebra is symmetrizable;
this is the class referred to as  Class I from Introduction
(classified in~\cite{Ksuper}).
Class (AFF) consists of Class II from Introduction (classified in~\cite{vdLeur})
together with the non-symmetrizable 
Kac-Moody superalgebras $\fq_n^{(2)}$ and $S(2|1;b)$,
described in~\cite{Hoyt},\cite{Sint}.
 Class (IND) consists of anisotropic (symmetrizable and non-symmetrizable) Kac-Moody superalgebras $\fg(A,\tau)$, where $A$ is of (Ind)-type according to the classification in~\cite{Kbook}, Theorem 4.3, together with the non-symmetrizable Kac--Moody superalgebras  $Q(m,n,s)^{\pm}$; see~\cite{K78}, \cite{Hoyt},\cite{Sint} for classifications.

\section{Siblings}\label{sect:Siblings}
We say that $\fg(A,\tau)$ and $\fg(A',\tau')$ are {\em siblings} if both are Kac--Moody superalgebras, $(A,\tau)$ is anisotropic, and $A=A'$. Each $\fg:=\fg(A,\tau)$ has a sibling $\fg':=\fg(A,\emptyset)$, which is a Kac--Moody algebra. We identify the Cartan subalgebra of $\fg'$ with $\fh$ and denote the root system of $\fg'$ by $\Delta'$. The root systems $\Delta$ and $\Delta'$ have the same set of simple roots $\Sigma$ and the same simple coroots $\{\alpha^\vee|\ \alpha\in\Sigma\}$. Therefore, siblings have the same Weyl group and the same real roots.

For example, for $A=(2)$ we have the pair of siblings
$$
\fg(A,\emptyset)=\fsl_2,\qquad
\fg(A,\{1\})=\mathfrak{osp}(1|2).
$$
More generally, $\mathfrak{osp}(1|2n)$ and $\mathfrak{o}_{2n+1}$ are siblings for every $n\geq1$. There are no other finite-dimensional siblings. The connection between these algebras was studied 
in~ \cite{Zhang}, \cite{Lan}, although the term ``siblings'' was not used there.

\subsection{}
\begin{thm}{thm:siblings}
Let $(A,\tau)$ be a symmetrizable anisotropic Cartan datum satisfying~(\ref{eq:cond1}). 
Let $\fg:=\fg(A,\tau)$ and $\fg':=\fg(A,\emptyset)$.
\begin{enumerate}
\item If $L_{\fg}(\lambda)$ is an integrable\footnote{For an anisotropic Kac-Moody superalgebra integrable modules are defined
as in the case for Kac-Moody algebras: we require all  Chevalley generators $e_i,f_i$, $i=1,\ldots,\ell$, to act locally nilpotently} irreducible $\fg$-module
of highest weight $\lambda$, then $L_{\fg'}(\lambda)$ is an integrable $\fg'$-module and
$$
\ch L_{\fg}(\lambda)=\ch L_{\fg'}(\lambda).
$$

\item
One has $\Delta^{\ima}=(\Delta')^{\ima}$ and 
$$
\begin{array}{ll}
\dim\fg_{\alpha}=\dim\fg'_{\alpha}  & \text{ for }\ \alpha\in\Delta^{\ima}\ \ \text{ s. t. }\ \ \frac{\alpha}{2}\not\in\Delta^{\ima}_{\ol{1}},\\
\dim \fg_{2\alpha}
=\dim \fg'_{\alpha}+\dim\fg'_{2\alpha}
& \text{ for }\ \alpha\in\Delta^{\ima}_{\ol{1}}.
\end{array}$$
\end{enumerate}
\end{thm}

\subsubsection{Remark}
The equality $\Delta^{\ima}=(\Delta')^{\ima}$ also follows from the general description of imaginary roots for Kac--Moody superalgebras given in~\cite{Shay}. Note that this equality implies
$
\Delta=\Delta'\cup
\{2\alpha|\ \alpha\in W(\Delta_{\ol{1}}\cap \Sigma)\},
$
where $W$ is the Weyl group.

\subsubsection{Example}

It is possible for $L_{\fg'}(\lambda)$ to be integrable while $L_{\fg}(\lambda)$ is not. For example, for $A=(2)$ we have $\fg=\mathfrak{osp}(1|2)$ and $\fg'=\fsl_2$. Since $\fg$ and $\fg'$ are finite-dimensional, integrability of simple modules is equivalent to finite-dimensionality. If $\alpha$ is the simple root, then $L_{\fg'}(\frac{\alpha}{2})$ is two-dimensional, whereas $L_{\fg}(\frac{\alpha}{2})$ is infinite-dimensional.

\subsubsection{Proof of~\Thm{thm:siblings}}
The module $L_{\fg}(\lambda)$ is integrable if and only if, for each simple root $\alpha$, the value
$\langle\lambda,\alpha^\vee\rangle
$
is a non-negative integer and is even if $\alpha$ is odd. Therefore, if $L_{\fg}(\lambda)$ is integrable, then $L_{\fg'}(\lambda)$ is integrable.

Since  $\Delta^+$ and $(\Delta')^+$
have the same set of simple roots $\Sigma$ and the same 
simple coroots,
we can choose the  Weyl vectors to be equal (i.e., $\rho=\rho'$). 
If $L_{\fg}(\lambda)$ is integrable, then, by above,  $L_{\fg'}(\lambda)$ is
integrable and the Weyl-Kac character formula (see~\cite{Kbook}, Theorem 10.4, and~\cite{K78}, Theorem 1) gives
\begin{equation}\label{eq:siblingWeylfor}\begin{array}{rl}
\prod_{\alpha\in (\Delta')^+} (1-e^{-\alpha})^{\dim \fg'_{\alpha}}\ch L_{\fg'}(\lambda)&=\sum\limits_{w\in W} \sgn w\ e^{w(\lambda+\rho)-\rho}\\
&=
\frac{\prod\limits_{\alpha\in \Delta_{\ol{0}}^+} (1-e^{-\alpha})^{\dim \fg_{\alpha}}}
{\prod\limits_{\alpha\in \Delta_{\ol{1}}^+} (1+e^{-\alpha})^{\dim \fg_{\alpha}}}\ch L_{\fg}(\lambda).\end{array}\end{equation}
For $\lambda=0$ this gives 
\begin{equation}\label{eq:siblingWeyldenom}\
\prod_{\alpha\in (\Delta')^+} (1-e^{-\alpha})^{\dim \fg'_{\alpha}}=
\frac{\prod\limits_{\alpha\in \Delta_{\ol{0}}^+} (1-e^{-\alpha})^{\dim \fg_{\alpha}}}
{\prod\limits_{\alpha\in \Delta_{\ol{1}}^+} (1+e^{-\alpha})^{\dim \fg_{\alpha}}}.\end{equation}
Combining~(\ref{eq:siblingWeylfor}) and~(\ref{eq:siblingWeyldenom}) we obtain 
$\ch L_{\fg}(\lambda)=\ch L_{\fg'}(\lambda)$ as required.
This establishes (i).

For (ii) observe that, by~(\ref{eq:galpha}), for 
$\alpha\in\Delta^{\ree}_{\ol{0}}$
we have $\dim \fg_{\alpha}=\dim\fg'_{\alpha}=1$, and for
$\alpha\in\Delta^{\ree}_{\ol{1}}$
we have $\dim \fg_{\alpha}=\dim\fg_{2\alpha}=1$
and $\dim \fg'_{\alpha}=1$, $\dim\fg'_{2\alpha}=0$. 
Therefore
$$\prod\limits_{\alpha\in (\Delta')^{\ree,+}} (1-e^{-\alpha})^{\dim \fg'_{\alpha}}=\prod\limits_{\alpha\in \Delta_{\ol{0}}^{\ree,+}} (1-e^{-\alpha})^{\dim \fg_{\alpha}}\frac{\prod\limits_{\alpha\in \Delta_{\ol{1}}^{\ree,+}} (1-e^{-2\alpha})^{\dim \fg_{\alpha}}}
{\prod\limits_{\alpha\in \Delta_{\ol{1}}^{\ree,+}} (1+e^{-\alpha})^{\dim \fg_{\alpha}}}.$$
Recall that 
$$\Delta_{\ol{0}}^+=\Delta_{\ol{0}}^{\ree,+}\coprod \{2\alpha|\ \alpha \in\Delta_{\ol{1}}^{\ree,+}\}\coprod \Delta_{\ol{0}}^{\ima,+}.$$
Using~(\ref{eq:siblingWeyldenom}) we get
$$\prod\limits_{\alpha\in (\Delta')^{\ima,+}} 
(1-e^{-\alpha})^{\dim \fg'_{\alpha}}=\frac{\prod\limits_{\alpha\in \Delta_{\ol{0}}^{\ima,+}} (1-e^{-\alpha})^{\dim \fg_{\alpha}}}
{\prod\limits_{\alpha\in \Delta_{\ol{1}}^{\ima,+}} (1+e^{-\alpha})^{\dim \fg_{\alpha}}}.$$

Set $X_1:=\Delta^{\ima,+}_{\ol{1}}$ and
$X_0:=\{\alpha\in \Delta^{\ima,+}\cup (\Delta')^{\ima,+}|\ \alpha,\frac{\alpha}{2}\not\in \Delta^{\ima,+}_{\ol{1}}\}$.
Then 
$$
\prod\limits_{\alpha \in X_0} (1-e^{-\alpha})^{m_{\alpha}}
\prod\limits_{\alpha \in X_1} (1-e^{-\alpha})^{m_{\alpha-}}
\prod\limits_{\alpha \in X_1} (1+e^{-\alpha})^{m_{\alpha+}}=1$$
 where $m_{\alpha}:=\dim \fg'_{\alpha}-\dim\fg_{\alpha}$ and
$$ \begin{array}{lcl}
m_{\alpha+}:=\dim\fg_{\alpha}+\dim \fg'_{2\alpha}-\dim\fg_{2\alpha}, & &
m_{\alpha-}:=\dim\fg'_{\alpha}+\dim \fg'_{2\alpha}-\dim\fg_{2\alpha}.
\end{array}$$

For every $\beta\in X_1$ we have $\beta,2\beta\not\in X_0$
by construction and $2\beta\not\in X_1$ by~$\S$~\ref{parity}.
Using~\Lem{lem:technicX}, we obtain
$m_{\alpha}=0$ for all $\alpha\in X_0$ and $m_{\alpha\pm}=0$ for all $\alpha\in X_1$.  This completes the proof.
\qed

\subsubsection{}
\begin{cor}{cor:>}
 Let $\fg$ be a symmetrizable  anisotropic Kac-Moody superalgebra.
If $\gamma$ is an odd imaginary root, then 
$\dim \fg_{2\gamma}>\dim \fg_{\gamma}$. 
\end{cor}
\begin{proof}
By~\Thm{thm:siblings} 
(ii),  we have
$\dim \fg_{2\gamma}-\dim \fg_{\gamma}=\dim\fg'_{2\gamma}$
and  $\gamma\in\Delta^{'\ima}$.
By~\cite{Kbook}, Chapter V, $\gamma\in\Delta^{'\ima}$ implies $2\gamma\in\Delta^{'\ima}$, so $\dim \fg'_{2\gamma}>0$.
\end{proof}

\subsection{}
\begin{lem}{lem:technicX}
Let $X$  be a subset of $Q^+:=\mathbb{Z}_{\geq 0}\Sigma$ and
$X=X_0\coprod X_1$ is a decomposition such that 
$2\beta\not\in X$ for every $\beta\in X_1$.
If 
$$\prod\limits_{\alpha\in X_0}(1-e^{-\alpha})^{m_{\alpha}}
\prod\limits_{\alpha\in X_1}(1-e^{-\alpha})^{m_{\alpha-}}
\prod\limits_{\alpha\in X_1}(1+e^{-\alpha})^{m_{\alpha+}}=1$$
for some integers $m_{\alpha}, m_{\alpha\pm}$, then
all these integers are equal to $0$.
\end{lem}
\begin{proof}
Delete all $\alpha$ whose exponents are all zero. Then
$m_{\alpha}\not=0$ for every $\alpha\in X_0$ and
$m_{\alpha+}\not=0$ or $m_{\alpha-}\not=0$ 
for every $\alpha\in X_1$. 

Denote the left-hand side of the above formula by $P(X,M)$, where $M$ stands for the set of exponents
$$
\{m_\alpha\}_{\alpha\in X_0}
\cup
\{m_{\alpha\pm}\}_{\alpha\in X_1}.
$$

We have to verify that $P(X,M)=1$ if $M$ is non-empty.
Suppose, to the contrary, that  $P(X,M)\not=1$ and $M$ is non-empty.

 For  each $\nu=\sum\limits_{\gamma\in \Sigma} n_{\gamma}\gamma\in Q^+$
we set $\htt \nu:=\sum\limits_{\gamma\in \Sigma} n_{\gamma}$
and introduce
$$X^{\leq i}:=\{\gamma\in X|\ 
\htt \gamma\leq i\};\ \ \ \ s:=\min_{\alpha\in X}\htt \alpha.$$
For all $i$ the sets $X^{\leq i}$ are finite
(since the sets
$\{\gamma\in Q^+|\ \htt \gamma=j\}$ are finite).

Let $\beta\in X$ be such that $\htt \beta=s$.
If $\beta\in X_0$, then the coefficient 
 of $e^{-\beta}$ in the expansion of $P(X;M)$ is equal to $-m_{\beta}$, so
 $m_{\beta}=0$, a contradiction. This gives
\begin{equation}\label{eq:Xs}
X^{\leq s}\subset X_1.\end{equation}

If $\beta\in X_1$, then the coefficient of $e^{-\beta}$ in the expansion of $P(X;M)$
is $m_{\beta_+}-m_{\beta-}$, so $m_{\beta_+}=m_{\beta_-}$. Then
$$P(X;M)=P(X';M'),\  \text{ 
where }X'_0=X_0\cup\{2\beta\},\ \ 
X'_1=X_1\setminus \{\beta\}$$ 
and 
$$m'_{2\beta}=m_{\beta+},
\ m'_{\alpha}=m_{\alpha},\ \ m'_{\alpha\pm}=m_{\alpha \pm}\ 
\text{ for all }\alpha\in X'\cap X.$$

Repeating the same procedure several times,
we obtain 
$$P(X,M)=P(Y,N),\ \text{ where
}\  Y_0=X_0\cup\{2\alpha|\ 
\alpha\in X_1\cap X^{\leq 2s}\},\  Y_1= X_1\setminus X^{\leq 2s}.$$
Letting $q:=\min_{\alpha\in Y}\htt \alpha$ we have
 $q\leq 2s$ and $Y^{\leq q}\subset Y_0$, which 
contradicts~(\ref{eq:Xs}) applied to $(Y,N)$.
\end{proof}

\subsection{Remark}
In~\cite{Kumar}, Shrawan Kumar established the Weyl-Kac denominator identity for general Kac--Moody algebras, without assuming symmetrizability of Cartan matrix. To the best of our knowledge, no analogous formula has yet been established for non-symmetrizable Kac--Moody superalgebras.

\section{Shapovalov determinant for  symmetrizable Kac-Moody superalgebras}\label{sect:Shap}
Let $\fg$ be a   symmetrizable Kac-Moody superalgebra.

\subsection{}
We introduce the Kostant partition function  $K: \mathbb{Z}\Delta\to \mathbb{Z}$ by the formula
$$K(\mu):=\dim\cU(\fn^+)_{\mu}.$$
Thus $K(\mu)\geq 0$ and $K(\mu)=0$ if
$\mu\not \in Q^+$, where $Q^+:=\mathbb{Z}_{\geq 0}\Sigma$.

For an odd root $\alpha$ the function
$$K_{\alpha}(\mu):=\sum\limits_{i=0}^{\infty} (-1)^{i} K(\mu-i\alpha)$$ 
is equal to the number of PBW-monomials of weight $\mu$
which does not contain $e_{\alpha}$. In particular, 
 $K_{\alpha}(\mu)\geq 0$ and $K_{\alpha}(\mu)=0$ if
$\mu\not \in Q^+$.

\subsubsection{}
The  Shapovalov determinants were computed 
for semisimple Lie  algebras in~\cite{Shap}, for
symmetrizable Kac-Moody algebras in~\cite{KK} and
for  symmetrizable  Kac-Moody superalgebras in~\cite{Kconf} (with the sign corrected in~\cite{K86}).
The Shapovalov determinant is defined only up to non-zero scalar multiples. Throughout, we disregard such scalar factors.

One has 
$$\det S_{\nu}=\prod\limits_{m=1}^{\infty}
\prod\limits_{\alpha\in \Delta^+} 
H_{m\alpha}^{(-1)^{(m+1)p(\alpha)}K(\nu-m\alpha)\dim\fg_{\alpha}}.$$

\subsection{}\label{Xi01}
We set
$$\begin{array}{ll}
\Xi_0:=\{(\alpha,m)|\ \alpha\in\Delta^{\ree+}_{\ol{0}}, \ m\in\mathbb{Z}_{>0}\},\ &
\Xi_{\is}:=\{(\alpha,1)|\ \alpha\in\Delta^{\ree+}_{\ol{1}}\cap\Delta^{\is}\}\\
\Xi_1:=\{(\alpha,2m-1)|\ \alpha\in\Delta^{\ree+}_{\ol{1}}\setminus\Delta^{\is}, \ m\in\mathbb{Z}_{>0}\},\ & 
\ \Xi:=\Xi_0\cup\Xi_1\cup\Xi_{\is}.
\end{array}$$
Notice that for $(\alpha,m), (\beta,n)\in \Xi$ we have 
$m\alpha=n\beta$ only if $(\alpha,m)=(\beta,n)$.

\subsubsection{}
If $\alpha\in\Delta^{\ree}$, then, by~(\ref{eq:galpha}),  $\dim \fg_{\alpha}=1$ and 
$\mathbb{Z}\alpha\cap \Delta=\{\pm\alpha\}$ except for the case when $\alpha$ is an odd non-isotropic root; in this case $\mathbb{Z}\alpha\cap \Delta=\{\pm\alpha;\pm 2\alpha\}$.
This gives 
$$\det S_{\nu}= S^{\ima}_{\nu}\cdot \prod\limits_{(\alpha,m)\in\Xi_0\cup\Xi_1}
H_{m\alpha}^{K(\nu-m\alpha)}\prod\limits_{(\alpha,1)\in\Xi_{\is}}
H_{\alpha}^{K_{\alpha}(\nu-\alpha)}$$
where
$$\begin{array}{rll}
S_{\ima}:=&\prod\limits_{m=1}^{\infty}&\prod\limits_{\alpha\in \Delta^{\ima+}_{\ol{0}}} H_{m\alpha}^{\dim\fg_{\alpha}K(\nu-m\alpha)}\prod\limits_{\alpha\in \Delta^{\ima+}_{\ol{1}}} H_{m\alpha}^{(-1)^{m+1}\dim\fg_{\alpha}K(\nu-m\alpha)}\\
=&\prod\limits_{m=1}^{\infty}&\prod\limits_{\alpha\in \Delta^{\ima+}_{\ol{0}}: \alpha/2\not\in\Delta_{\ol{1}}^{\ima}} H_{m\alpha}^{\dim\fg_{\alpha}K(\nu-m\alpha)}\\
& &\ \ \prod\limits_{\alpha\in \Delta^{\ima+}_{\ol{1}}}   H_{2m\alpha}^{(\dim \fg_{2\alpha}-\dim\fg_{\alpha})
K(\nu-2m\alpha)}
H_{(2m-1)\alpha}^{\dim\fg_{\alpha}
K(\nu-(2m-1)\alpha)}
\end{array}$$

We claim that the exponent of the linear factors $H_{j\alpha}$ in the above formula are non-negative,
equal to zero if $\nu\not\in j\alpha+Q^+$ and are
positive for $\nu=j\alpha$. 
 The only point that remains to be verified is
$$\dim \fg_{2\gamma}>\dim\fg_{\gamma} \text{ for all }\gamma\in \Delta^{\ima+}_{\ol{1}}.$$ 
Indeed, \Cor{cor:>} gives the inequality in the case when $\fg$
is  anisotropic. If $\fg$ is finite-dimensional,
then $\Delta^{\ima}$ is empty.
By  Hoyt-Serganova classification~\cite{Hoyt},\cite{Sint}
an infinite-dimensional symmetrizable Kac-Moody superalgebra that is 
not anisotropic is affine.  In the latter case, by~\cite{vdLeur}, $\Delta^{\ima}_{\ol{1}}$ is non-empty only if $\fg$ is of type $A(2m|2n)^{(4)}$. In this case, for any $\gamma\in\Delta^{\ima}_{\ol{1}}$, one has
$\dim\fg_\gamma=1$ and
$\dim\fg_{2\gamma}=2m$ if $m=n$, and $\dim\fg_{2\gamma}=m+n+1$ otherwise; see, for example,~\cite{Peleg}. Thus, the inequality holds.

\subsubsection{}
We conclude that  $M(\lambda)$ is irreducible if and only if
 $H_{m\alpha}(\lambda)\not=0$
for all $(\alpha,m)\in \Xi$ and $(\alpha,m)\in\Delta^{\ima+}\times\mathbb{Z}_{>0}$.

 Kac's description of imaginary roots 
(\cite{K78} Proposition 1.6 
extended in~\cite{Shay} to all Kac-Moody 
superalgebras) implies that 
 $\alpha\in\Delta^{\ima+}$
forces $m\alpha \in\Delta^{\ima+}$. This gives the following corollary.

\subsection{}
\begin{cor}{cor:Vermairred}
Let $\fg$ be a symmetrizable Kac-Moody superalgebra.
Then 
 $M(\lambda)$ is irreducible if and only if
 $H_{m\alpha}(\lambda)\not=0$
for all $(\alpha,m)\in \Xi$ and $(\alpha,1)$ for
$\alpha\in\Delta^{\ima+}$.
\end{cor}

The criterion stated in the Introduction is a reformulation of the above corollary.

\subsection{Useful formula}\label{useful}
The above formula for the Shapovalov determinant may contain proportional linear factors; here we collect them and rewrite the determinant as a product over a reduced index set $\ol{A}$. We let
$$\begin{array}{ll}
\ol{\Delta}^+:&=\{\alpha\in\Delta^+|\ \alpha/j\not\in\Delta^+\ \text{ for all } j\in\mathbb{Z}_{>1}\};\\
\ol{A}:=&\{(\alpha,m)|\ \alpha\in \ol{\Delta}^+,
\ m\in\mathbb{Z}_{>0},\\
&\ \  \text{$m=1$ if $(\alpha,\alpha)=0$, and $m$ is odd if $\alpha$
is  odd and $(\alpha,\alpha)\not=0$}\}.
\end{array}$$

Then
\begin{equation}\label{eq:detred}
\det S_{\nu}=\prod\limits_{(\alpha,m)\in \ol{A}} H_{m\alpha}^{d_{m\alpha}(\nu)},\end{equation}
where $d_{m\alpha}(\nu)\geq 0$, $d_{m\alpha}(m\alpha)=1$,
and  $d_{m\alpha}(\nu)=0$ if $\nu\not\in m\alpha+Q^+$.
The linear factors in the above formula are not proportional.

\subsection{Blocks in category $\CO(\fg)$}\label{blocks}
The description  below  slightly corrects~\cite{GSsnow}, 1.13,  where it was derived from an incorrect irreducibility criterion for Verma modules; here it is based on~\Cor{cor:Vermairred}.

We consider the category  $\CO(\fg)$ introduced in~\cite{DGK}:
this is the full subcategory of the category of $\fg$-modules $N$ satisfying
the following conditions
\begin{enumerate}
  \item $\fh$ acts diagonally with  finite-dimensional
eigenspaces,
\item there exist a finite set $\{\lambda_i\}_{i=1}^s\subset\fh^*$
such that $N_{\nu}\not=0$ implies $\lambda_i-\nu\in Q^+$ for some $i$.
  \end{enumerate}

The category $\CO(\fg)$ admits a block decomposition 
in the following sense.
Consider the  graph whose vertices are weights, with two vertices $\lambda$ and $\nu$
 connected by an edge if
$\Ext^1(L(\lambda),L(\nu))\not=0$.
Define the equivalence relation $^\sim_{\CO}$
on $\fh^*$ by  declaring
$\lambda^\sim_{\CO} \nu$ if and only if $\lambda,\nu$ lie in the same component of this graph. The arguments of~\cite{DGK}, originally developed for the non-super case, imply
 that $\CO(\fg)$ admits a block decomposition, which means
that for every  module $N\in \CO(\fg)$ 
there exists a unique set of submodules $\{M_{\chi}\}_{\chi\in \fh^*/\\ ^{\sim}_{\CO}}$ such that
\begin{itemize}
\item[$\bullet$]
$N=\oplus_{\chi\in \fh^*/\ ^\sim_{\CO}} M_{\chi}$;
\item[$\bullet$]
for each $\chi\in \fh^*/\ ^\sim_{\CO}$ every irreducible subquotient
of $M_{\chi}$ is of the form $L(\lambda)$ for $\lambda\in\chi$.
\end{itemize}

\subsubsection{}
The relation $^\sim_{\CO} $ can be explicitly described as follows. 

Retain notation of~\ref{useful} and set
$$\begin{array}{ll}
\cK:=&\{(\lambda,\lambda-m\alpha)\in \fh^*\times\fh^*|\ 
(\alpha,m)\in\ol{A},\  2(\lambda+\rho,\alpha)=m(\alpha,\alpha)\}
\end{array}$$

Denote by $\succeq$  the  reflexive and transitive relation on $\fh^*$ generated by $\cK$  (this means that 
$\succeq$ is the minimal reflexive and transitive relation  on $\fh^*$
such that  $(\lambda,\nu)\in \cK$ implies
$\lambda\succeq\nu$).

From~\cite{KK}, Thm. 2 and Prop. 4.1 (the proofs go through verbatim for superalgebras)
\begin{equation}\label{eq:subquo}
[M(\lambda):L(\nu)]\not=0\ \Longrightarrow\ \lambda\succeq\nu.\end{equation}

\subsubsection{}
\begin{prop}{prop:cO}
\begin{enumerate}
\item $\Hom(M(\lambda-m\alpha),M(\lambda))\not=0$
if $(\alpha,m)\in\ol{A}$
and  $2(\lambda+\rho,\alpha)=m(\alpha,\alpha)$.
\item The relation $^{\sim}_{\CO}$ is the equivalence relation  on $\fh^*$ generated by $\cK$.  
\end{enumerate}
\end{prop}
\begin{proof}
  For $(\alpha,m)\in \ol{A}$ we let  
$$\begin{array}{l}
\gamma_{m\alpha}:=\{\lambda\in\fh^*|\ H_{m\alpha}(\lambda)=0\},\\
\mathring{\gamma}_{m\alpha}:=\{\lambda\in \gamma_{m\alpha}|\ \lambda\not\in
\gamma_{s\beta}\ \text{ for all } (\beta,s)\in\ol{A}\ \text{ s. t. }s\beta<m\alpha\}.\end{array}$$

Let  $\lambda\in \mathring{\gamma}_{m\alpha}$.
By~(\ref{eq:detred}),
 $\det S_{\nu}(\lambda)\not=0$ for $\nu<m\alpha$ and
$\det S_{m\alpha}(\lambda)=0$. Therefore
$M(\lambda)_{\lambda-m\alpha}$ contains a primitive vector.
The set $\mathring{\gamma}_{m\alpha}$ is  Zariski dense in 
$\gamma_{m\alpha}$. 
The density arguments of~\cite{BGG}
imply $\Hom(M(\lambda-m\alpha),M(\lambda))\not=0$
for every $\lambda\in {\gamma}_{m\alpha}$ if $(\alpha,m)\in\ol{A}$
and  $(\alpha,\alpha)\not=0$ or $\alpha/j\not\in\Delta^+$ 
for all $j>1$ (since the condition "$M(\lambda)_{\lambda-m\alpha}$ contains a primitive vector” is Zariski closed in $\lambda$, because it is a rank condition on matrices polynomial in $\lambda$). This proves (i).

Now (ii) follows from (i) and~(\ref{eq:subquo}),
since, for $\lambda\not=\nu$, one has $\Ext^1(L(\lambda),L(\nu))=0$ 
if  $[M(\lambda):L(\nu)]=[M(\nu):L(\lambda)]=0$.
\end{proof}

\subsubsection{}
 By~\cite{GSsnow}, Lemma 1.13.4, whose proof goes through verbatim with 
$\cK$ as above, if $\lambda^{\sim}_{\CO} \nu$, then
$\lambda$ and $\nu$ are either both critical or both non-critical, and likewise are either both typical or both atypical. This allows us to distinguish typical, atypical, critical, and non-critical blocks in $\CO(\fg)$.

\section{Further results on Shapovalov determinants}
Various aspects and generalizations of Shapovalov determinants and Shapovalov elements for Lie superalgebras and related algebras were studied in many papers including~\cite{Shap},\cite{Kconf},\cite{KK},\cite{K86},
\cite{HY}, \cite{LL},\cite{Mus1},\cite{Gcentre} and~\cite{GSq2}.
 In this section, we collect several further results on Shapovalov determinants that apply beyond the symmetrizable case and briefly discuss what is known for the non-symmetrizable affine Kac-Moody superalgebras.

\subsection{Notation}
For $\alpha\in\Delta^{\ree}$ we choose a non-zero vector
$\alpha^{\vee}\in\fg(\alpha)\cap\fh$ and define
$H_{m\alpha}$ by
$$H_{m\alpha}(\lambda):=\langle \lambda+\rho, m\alpha^{\vee}\rangle-
\frac{m^2}{2}\langle \alpha, \alpha^{\vee}\rangle.
$$
In the symmetrizable case, this agrees with our previous notation.

We write
$$\det S_{\nu}= S^{\ima}_{\nu}\cdot \prod\limits_{(\alpha,m)\in\Xi}
H_{m\alpha}^{d_{m\alpha}(\nu)},$$
where $S^{\ima}_{\nu}$ contains none of the irreducible factors $H_{m\alpha}$ with $(\alpha,m)\in\Xi$. (For symmetrizable case, $d_{m\alpha}=K(\nu-m\alpha)$ for $(\alpha,m)\in \Xi_0\cup\Xi_1$
and $d_{\alpha}=K_{\alpha}(\nu-\alpha)$
for $(\alpha,1)\in\Xi_{\is}$.)
\subsection{Leading term}\label{leading}
For each $\alpha\in\Delta^+$, let $D_{\alpha}$
be the matrix of the pairing $\fg_{\alpha}\times\fg_{-\alpha}\to\fh$
given by $(x,y)\mapsto [x,y]$. 
We regard $\det D_{\alpha}$ as an element of $S(\fh)$.
Like the Shapovalov determinant, $\det D_\alpha$ is understood up to a non-zero scalar multiple. For example,
$\det D_{\alpha}=\alpha^{\vee}$ if $\alpha$ is a real root.
In the symmetrizable case,~\cite{Kbook}, Theorem 2.2, gives
$\det D_{\alpha}=h_{\alpha}^{\dim\fg_{\alpha}}$,
where $h_{\alpha}$ is the image of $\alpha$ under the
isomorphism $\fh^*\iso \fh$ induced by the form $(-,-)$.

The following is~\cite{GSPoisson}, Theorem 4.1; the non-vanishing assumption, which is automatic in the Poisson case considered there, was inadvertently omitted from its statement.

\subsubsection{}
\begin{cor}{cor:leading}
Assume that $\det D_\alpha\neq 0$ for all $\alpha\in\Delta^+$.
Then the leading term of $\det S_\nu$ is
$$
\prod\limits_{\alpha\in\Delta^+}(\det D_{\alpha})^{d'_{\alpha}},\ \ \ 
d'_{\alpha}:=
\sum\limits_{m=1}^{\infty}
(-1)^{(m+1)p(\alpha)}K(\nu-m\alpha).
$$
In particular, $\det S_{\nu}\not=0$ for all $\nu$.
\end{cor}

\subsection{}
\begin{prop}{prop:dmalpha}
Let $\lambda$ be such that
$H_{m\alpha}(\lambda)=0$ for some $(\alpha,m)\in\Xi_0\cup\Xi_1$, and
$H_{n\beta}(\lambda)\neq0$
for all $(\beta,n)\in\Delta^{\ree+}\times\mathbb Z_{>0}$ with $\beta\neq\alpha$.

\begin{enumerate}
\item If $\alpha$ is non-isotropic, then $M(\lambda)$ contains a primitive vector of weight $\lambda-m\alpha$, unique up to a non-zero scalar multiple. This vector generates a submodule isomorphic to $M(\lambda-m\alpha)$.

\item If $\alpha$ is isotropic, then $M(\lambda)$ contains a submodule $N$ such that
$$
(1+e^{-\alpha})\ch N=\ch M(\lambda-\alpha).
$$
\end{enumerate}
\end{prop}
\begin{proof}
We sketch the proof, omitting some technical details.

We denote by $\vac$ the highest weight vector of $M(\lambda)$.
The {\em spine} of $\fg$ is the set $\Sp$
of all bases of $\Delta$, obtained 
from $\Sigma$ by chains of  reflexions
consisting of isotropic roots.
By~\cite{GHS}, Prop. 4.3.13, $W$ is a Coxeter group 
generated by $\{s_{\gamma}\}_{\gamma\in\pi}$, where
 $$\pi:=\bigl(\cup_{\Sigma'\in \Sp} \Sigma'\bigr)
 \setminus \Delta^{\is};$$
 this set is called the set of principal roots.

 Recall that $\alpha$ is a real root.
By~\cite{GHS}, Cor. 5.2.9, there exists $w\in W$
and $\Sigma'\in\Sp$ such that
 $w\alpha\in \Sigma'$. We choose $w$ to 
be of the minimal length and prove the assertions by induction on this length.

Consider the case when $w=1$, that is $\alpha\in\Sigma'$
for some $\Sigma'\in\Sp$. We view $\Sp$ as a subgraph of the skeleton 
with the vertices corresponding to bases and the  edges marked by isotropic reflexions.  Choose a base $\Sigma'\in\Sp$ containing $\alpha$ and having minimal distance from $\Sigma$,
and let 
$$\Sigma=\Sigma_0\overset{\beta_1}{\to}\Sigma_{1}\overset{\beta_2}{\to}\ldots
\overset{\beta_s}{\to}\Sigma_{s}=\Sigma'$$
be a  shortest path 
from $\Sigma$ to $\Sigma'$.
If $\Sigma=\Sigma'$, then $\alpha=\alpha_i\in\Sigma$, and $v=f_i^m \vac$
is a primitive  vector, so the assertions hold. 

Let $\Sigma'\not=\Sigma$.
Recall that $\beta_i\in\Sigma_{i-1}$, so $\alpha\not\in\{\beta_i\}_{i=1}^s$ (since $\Sigma'$ has a minimal distance from $\Sigma$).
Moreover,
$$\Delta^{+}(\Sigma_j)=\Delta^+\setminus\{\beta_i\}_{i=1}^j\cup\{-\beta_i\}_{i=1}^j.$$
Since $\alpha\in\Delta^+$ and the chosen path is shortest, we have $\alpha\in\Delta^+(\Sigma_i)$ for $i=1,\ldots,s$. Hence 
$\beta_i\not=\pm\alpha$ and
the assumption on $\lambda$ gives $$\langle\lambda+\rho,\beta_i^{\vee}\rangle\not=0\  \text{
for }i=1,\ldots,s.$$ 
Therefore,
$$
M(\lambda)=M_{\Sigma'}(\lambda'),\ \text{ for }\lambda'+\rho'=\lambda+\rho,
$$
where  $M_{\Sigma'}(\lambda')$ denotes the Verma module of highest weight $\lambda'$ with respect to the triangular decomposition determined by $\Sigma'$.
Since $\alpha\in\Sigma'$,
$M_{\Sigma'}(\lambda')$ contains a $\Sigma'$-primitive vector 
$v'$ of weight $\lambda'-m\alpha$, which is unique up to a multiplication by non-zero scalar. Let $N$ be the submodule
generated by $v'$.

If $\alpha$ is anisotropic, then $N=M_{\Sigma'}(\lambda'-m\alpha)$.
One has
$$\langle \lambda'-m\alpha+\rho',\beta_i^{\vee}\rangle=
\langle \lambda-m\alpha+\rho,\beta_i^{\vee}\rangle=
\langle s_{\alpha}(\lambda+\rho),\beta_i^{\vee}\rangle=
\langle \lambda+\rho,s_{\alpha}\beta_i^{\vee}\rangle
\not=0,$$
since for $\beta:=s_{\alpha}\beta_i\in\Delta^{\is}$
one has $s_{\alpha}\beta_i^{\vee}\in \mathbb{C}^*\beta^{\vee}$, 
by~\ref{reflections}, and $\langle \lambda+\rho,\beta^{\vee}\rangle\not=0$ by the assumptions
(since $H_{1,\beta}(\lambda)\not=0$).

If $\alpha$ is isotropic, then
$$(1+e^{-\alpha})\ch N=\ch M_{\Sigma'}(\lambda'-\alpha)=
e^{-\alpha}
\ch M_{\Sigma'}(\lambda')=e^{-\alpha}
\ch M(\lambda)=\ch M(\lambda-\alpha)$$
as required. This completes the proof for $w=1$.

Now consider the case $l(w)>0$. Then $w=w_1s_{\gamma}$
for $\gamma\in\pi$ and $w_1\in W$ such that $l(w_1)<l(w)$.
Choose a base $\Sigma'\in\Sp$ containing $\gamma$ and having minimal distance from $\Sigma$, and let 
$$\Sigma=\Sigma_0\overset{\beta_1}{\to}\Sigma_{1}\overset{\beta_2}{\to}\ldots
\overset{\beta_s}{\to}\Sigma_{s}=\Sigma'$$
be a  shortest path 
between $\Sigma$ and $\Sigma'$. Since $l(w)>1$, $\alpha\not\in\Sigma''$
for all $\Sigma''\in \Sp$. In particular,
$\beta_i\not=\pm\alpha$ for all $i=1,\ldots,s$ and $\gamma\not=\pm\alpha$.
As above, this gives, $M(\lambda)=M_{\Sigma'}(\lambda')$
where $\lambda+\rho=\lambda'+\rho'$. Moreover, the assumptions
on $\lambda$ imply
$\langle \lambda+\rho,\gamma^{\vee}\rangle\not\in\mathbb{Z}\langle \gamma^{\vee},\gamma\rangle$.
The induction step now follows from the Enright functor arguments of~\cite{GSsnow}, Theorem 3.2 and~\cite{GK25}, Section 5.7.
\end{proof}

\subsubsection{}
\begin{cor}{cor:dmu}
Let $\mathbb{C}\alpha^{\vee}\not=\mathbb{C}\beta^{\vee}$ for all distinct $\alpha,\beta\in\Delta^{\ree+}$.
\begin{enumerate}
\item If $(\alpha,m)\in \Xi$ (see~\ref{Xi01} for notation), then
$\Hom(M(\lambda-m\alpha), M(\lambda))\not=0$.
\item
Every Shapovalov determinant
$\det S_{\nu}$ is divisible by $$\prod\limits_{(\alpha,m)\in\Xi_0\cup\Xi_1}
H_{m\alpha}^{K(\nu-m\alpha)}\prod\limits_{(\alpha,1)\in\Xi_{\is}}
H_{\alpha}^{K_{\alpha}(\nu-\alpha)}.$$
\end{enumerate}
\end{cor}
\begin{proof}
The proof follows the standard arguments based on Zariski density (cf.,~\cite{BGG} or  the proof of~\Prop{prop:cO} above),
in the hyperplane $H_{m\alpha}=0$, of the set of $\lambda$ satisfying the assumptions  of~\Prop{prop:dmalpha}. 
\end{proof}

 \subsubsection{Remark}
The assumption of this corollary holds in symmetrizable case
and for $S(2|1,b)$, but does not hold
for  $\fq_n^{(2)}$.

\subsection{Non-symmetrizable cases}
For Kac-Moody superalgebras of class (IND), it is not known whether $\det S_\nu$ vanishes identically for some $\nu$. To the best of our knowledge, the corresponding question is open even for Kac-Moody algebras of class (IND). 

\subsubsection{Affine case}
Recall that 
the non-symmetrizable affine Kac-Moody superalgebras are $\mathfrak{q}_n^{(2)}$ and $S(2|1,b)$. 
We denote by $\delta$
the minimal imaginary root (the positive imaginary roots are
$i\delta$ for $i\in\mathbb{Z}_{>0}$).

For $\mathfrak{q}_n^{(2)}$, 
the Shapovalov determinants $\det S_{\nu}$ are identically zero 
for $\nu\geq 2\delta+Q^+$, see~\cite{GSq2}.
As a result, there are no irreducible Verma modules.
From the point of view of the preceding results, this algebra is exceptional in two respects: $\det D_{2s\delta}=0$ (while $\det D_\delta\neq0$), and $\alpha^\vee$ and $(\alpha+2\delta)^\vee$ are proportional for every real root $\alpha$. 
Thus, the assumptions of both~\Cor{cor:leading} and~\Cor{cor:dmu}
fail.

The non-symmetrizable affine Kac-Moody superalgebras $S(2|1,b)$
was studied in~\cite{Dhillon}. These algebras
are deformations of $\fsl(2|1)^{(1)}$ and are isomorphic to superconformal 
algebras $N=4$
studied in~\cite{KacLeur},\cite{STP}; the Shapovalov determinants of
were computed in~\cite{KW04}.
In this case, the assumptions of~\Cor{cor:leading}
and~\Cor{cor:dmu} hold; this gives
$$\det S_{\nu}=
\prod\limits_{(\alpha,m)\in\Xi_0\cup\Xi_1}
H_{m\alpha}^{K(\nu-m\alpha)}\prod\limits_{(\alpha,1)\in\Xi_{\is}}
H_{\alpha}^{K_{\alpha}(\nu-\alpha)}\prod\limits_{n=1}^{\infty} Q_{n\delta}^{d_{n\delta}(\nu)}$$
where $Q_{n\delta}\in S(\fh)$ are distinct irreducible quadratic polynomials with $d_{n\delta}(n\delta)=1$   
(and $d_{n\delta}(\nu)=0$ for $\nu\not\in n\delta+Q^+$).
This is consistent with the formula 
obtained in~\cite{KW04}.

\section*{
Declarations}

\subsection*{Ethics approval and consent to participate.}
Not applicable.

\subsection*{Consent for publication}
Not applicable. This manuscript contains no individual personal data, images, or case details requiring consent for publication.

\subsection*{Conflict of Interest}
The author declare that there is no conflict of interest.

\subsection*{Data availability}
Data sharing not applicable to this article as no datasets were generated or analysed
during the current study.

\subsection*{
Availability of data and material}
Not applicable. This article contains no data or supplementary material.

\subsection*{
Competing interests}
The author declares no competing interests.

\subsection*{Funding}
This research was supported by the Minerva Foundation with funding from the Federal German Ministry for Education and Research.
\subsection*{
Authors’ contributions}
M. Gorelik is the sole author of this manuscript. The author read and approved the final manuscript.

\subsection*{Acknowledgments}
The author is grateful to Vladimir Hinich, Victor Kac, and Vera Serganova
for helpful discussions and comments, and to the members of the STARS seminar
for their interest and useful feedback during the presentation of these results.
The author used ChatGPT (OpenAI) to assist with polishing the English and improving
the exposition of the manuscript; the mathematical content and arguments are the author's own.

\end{document}